\documentclass[11pt]{amsart}

\usepackage{amsfonts}
\usepackage{amsmath}
\usepackage{amsthm}
\usepackage{amssymb}
\usepackage{mathtools}
\usepackage{xcolor}
\usepackage{setspace}

\usepackage[
left=1in,
right=1in
]{geometry}

\usepackage{charter}
\usepackage{euler}

\usepackage{hyperref}
\usepackage{cleveref}
\hypersetup{hypertexnames = false, colorlinks, linkcolor = magenta, citecolor = cyan, urlcolor = blue, pdfstartview={XYZ null null 1}}

\newcommand\QQ{\ensuremath{\mathbb{Q}}}
\newcommand\ZZ{\ensuremath{\mathbb{Z}}}

\newcommand\PP{\ensuremath{\mathbb{P}}}
\newcommand\FF{\ensuremath{\mathbb{F}}}
\newcommand\pp{\ensuremath{\mathfrak{p}}}

\newcommand\mm{\ensuremath{\mathfrak{m}}}
\newcommand\oo{\ensuremath{\mathcal{O}}}
\newcommand\xx{\ensuremath{\mathcal{X}}}
\newcommand\Per{\ensuremath{\mathcal{P}}}

\newcommand\LL{\ensuremath{\mathcal{L}}}
\newcommand\TT{\ensuremath{\mathcal{T}}}

\newcommand{\set}[1]{\left\{#1\right\}}

\DeclareMathOperator\End{End}
\DeclareMathOperator\Spec{Spec}
\DeclareMathOperator\modulo{mod}
\DeclareMathOperator\ns{ns}

\newtheorem{theorem}{Theorem}[section]
\newtheorem{corollary}[theorem]{Corollary}
\newtheorem{lemma}[theorem]{Lemma}
\newtheorem{proposition}[theorem]{Proposition}
\newtheorem*{question}{Question}

\theoremstyle{remark}
\newtheorem*{remark}{Remark}

\begin{document}
\onehalfspacing

\title[On boundedness of periods of self maps]{On boundedness of periods of self maps of algebraic varieties}

\author{Manodeep Raha}
\address{School of Mathematics, Tata Institute of Fundamental Research, Mumbai, India 400005}
\curraddr{}
\email{rahamono@math.tifr.res.in}
\thanks{}

\date{\today}



\maketitle

\begin{abstract}
Let $X$ be an algebraic variety over a field $K \subset \overline{\QQ_p}$ and $f$ be a self map. When $K$ is a local field, the boundedness of $f$-periods in $X(K)$ is a well studied question. We will study the same question for certain infinite extensions over $\QQ_p$ under some conditions.
\end{abstract}

\section{Introduction}

Let $S$ be a set and $f : S \to S$ be a map. A point $P \in S$ is called a $f$-\emph{periodic} point if $f^n(P) = P$ for some positive integer $n$. The smallest such number is called the $f$-\emph{period} of $P$. Define $\Per (S, f)$ be the set of all possible $f$-periods, i.e.,
\begin{align*}
\Per (S, f) &= \set{n \mid n \text{ is a } f \text{-period of some } P \in S}
\end{align*}

In \cite{naf}, Fakhruddin proved :

\begin{theorem}
\label{naf_thm} 
Let $p$ be a prime and $\oo$ be the ring of integers of a local field $F$ over $\QQ_p$. For a proper variety $\xx$ over $\Spec (\oo)$, there exists a constant $m > 0$ such that $\Per (\xx (F), f) \le m$ for any $\oo$-morphism $f : \xx \to \xx$.
\end{theorem}

For a prime $p$, define $\Per_{(p)} (S, f)$ be the \textit{prime-to-$p$} part of $\Per (S, f)$, i.e., $$\Per_{(p)} (S, f) = \Per (S, f) \cap \set {n \, | \, n \text{ is coprime to } p}.$$ Then we show :

\begin{theorem}
\label{thma}
Let $\oo$ be the ring of integers of a local field $F$ over $\QQ_p$ and $K$ be a totally ramified infinite algebraic extension over $F$. For a proper variety $\xx$ over $\Spec (\oo)$, there exists a constant $m > 0$ such that $\Per_{(p)} (\xx (K), f) \le m$ for any $\oo$-morphism $f : \xx \to \xx$.
\end{theorem}

As a corollary, we get :

\begin{theorem}
\label{thmb}
Let $\oo$ be the ring of integers of a number field $F$ and $K$ be an infinite algebraic extension over $F$ which is totally ramified at a prime $\pp$ dividing $p$. For a proper variety $\xx$ over $\Spec (\oo)$, there exists a constant $m > 0$ such that $\Per_{(p)} (\xx (K), f) \le m$ for any $\oo$-morphism $f : \xx \to \xx$.
\end{theorem}

Let $A$ be an abelian variety with potential good reduction defined over a local field $F$ over $\QQ_p$. The finiteness of the set of torsion points in $A(F(\mu_{p^{\infty}}))$ is proved by Serre \cite{serre} and Imai \cite{imai} independently where $\mu_{p^{\infty}}$ is the group of $p$-power roots of unity. As a consequence of the previous theorem, we can get a weaker generalization :

\begin{theorem}
\label{ablian}
Assuming the setup of \autoref{thma}, let $A$ be an abelian variety defined over $\Spec (\oo)$. We define $\TT(A, K)$ be the set of all primes $\ell$ such that the group of $K$-rational $\ell$-torsion points $A(K)[\ell]$ is non trivial, i.e, $$\TT(A, K) = \set { \ell \text{ prime } | \, A(K)[\ell] \ne 0}$$ then the Dirichlet density of $\TT(A, K)$ is $0$. Moreover, $\TT(A, K) \cap \set{ \ell | \ell \not \equiv 1 (\modulo p) }$ is finite.
\end{theorem}

For the projective $n$-space $\PP^n$ over any number field $F$, the finiteness of $f$-periodic points of $\PP^n(F)$ was studied by Northcott \cite{northcott}. For abelian varieties, torsion points can be thought as periodic points for different isogenies. In this case, the finiteness of torsion points over a number field is well known \cite{aec}. For abelian varieties some results are known also over infinite extensions of $\QQ$. For an abelian variety $A$ over a number field $F$, the finiteness of torsion points in $A(F^{\text{cyc}})$ was proved by Ribet \cite{ribet} where $F^{\text{cyc}} = F \QQ^{\text{ab}}$ is the cyclotomic closure of $F$. For further details, we refer readers to \cite{silverman_book} and references therein.

\subsection*{Acknowledgement}
I would like to express my sincere gratitude to Prof. C.S. Rajan. I am also indebted to Prof. Najmuddin Fakhruddin, for introducing me to this problem during a conference at the Kerala School of Mathematics. Furthermore, I would like to thank to Mr. Niladri Patra, Dr. Pritam Ghosh, Mr. Sagar Shrivastava, and Dr. Sourav Ghosh. I am grateful to Ashoka University for providing me with an excellent research environment.

\section{Main lemma}
\label{prelim}

The following lemma is a slight modification of results by Fakhruddin \cite{naf} and Huang \cite[Theorem 1.2]{huang}.

\begin{lemma}
\label{naf_lemma}
Let $p$ be a prime and $\oo$ be a discrete valuation ring over $\ZZ_p$ with residue field $k$. For a separated scheme $\xx$ over $\Spec (\oo)$, there exists a constant $m > 0$ a such that $\Per_{(p)} (\xx(\oo)) \le m$ for any $\oo$-morphism $f : \xx \to \xx$. Moreover, $m$ depends only on the special fibre $\xx \times_{\Spec(\oo)} \Spec(k)$ of $\xx$ over $\Spec(k)$.
\end{lemma}

\begin{proof}
Let $P \in \xx(\oo)$ be a $f$-periodic point of period $n$. We denote the special fibre over $\Spec (k)$ by ${\xx}_s$ and the reduced morphism by $f$ only. Let ${P}_s \in {\xx}_s(k)$ be the reduction of $P$. ${P}_s$ is also a periodic point. Suppose $m$ is the period of ${P}_s$.

Clearly $f^m (P)$ restricted to the special fibre is a fixed point. Let $Z$ be the reduced Zariski closure of the $g$-orbit of $P$ where $g = f^m$. Hence, $Z$ is finite over $\Spec(\oo)$ with a unique closed point. So, $Z = \Spec (A)$ for some finite local $\oo$-algebra $A$. As $A$ is reduced, it is torsion free $\oo$-module.

As $g$ preserves the orbit and $Z$ is reduced, $g$ induces a map $Z \to Z$. So, it induces an $\oo$-linear map from $A$ to $A$. But as $g^n$ is the identity on $A$, hence $g \in \text{Aut} (A)$ and is of finite order.

Let $\mm$ be the maximum ideal of $A$. By \cite[Section 3]{huang}, $\dim_k (\mm / \mm^2) \le d + 1$ where $d$ is the maximum dimension of cotangent spaces of points over the special fibre $\xx_s$. There exists $r$ such that $g^r$ induces identity on $k$-vector space $\mm / \mm^2$. By \cite[Proposition 1]{naf}, the order of $g^r$ is a $p$-power, say $p^t$. It can be shown that $r$ is bounded by $(\# k)^d - 1$ \cite[Section 3]{huang} and $t$ is bounded by $v (p)$ \cite[Proposition 3]{naf} where $v$ is the valuation over $\oo$.

So, $n$ is bounded by $$N \cdot ((\# k)^d - 1) \cdot p^{v(p)}$$ where $N = \# \xx_s (k)$ \cite[Theorem 1.2]{huang}. But also if $n$ is coprime to $p$, then it is bounded by $$N \cdot ((\# k)^d - 1).$$ 
\end{proof}

\begin{remark}
If we consider $\xx$ is proper, then the same statement can be said about $K$-rational points where $K$ is the fraction field of $\oo$. Here the existence of reduction will be guaranteed by properness \cite[Chapter II.4]{hartshorne}. Also if the special fibre $\xx_s$ is non-singular then $d$ is the same as the dimension of $\xx_s$ \cite{huang}.
\end{remark}

\section{Proofs}
\label{ramified}

\subsection{Proof of the main theorem} 
Assuming the setup of \autoref{thma}, let $\oo$ be the ring of integers of $F$ with residue field $k$. Suppose $F'$ is a totally ramified finite extension over $F$. Let $\oo'$ be the ring of integers of $F'$. Clearly $\oo'$ has residue field $k$ also.

We consider $\xx$ is a separated scheme of finite type over $\Spec (\oo)$ and $\xx'$ is the base change of $\xx$ over $\Spec (\oo')$. Let us denote the special fibres of $\xx$ and $\xx'$ by $\xx_s$ and $\xx'_s$ respectively. The following fact is well known : 

\begin{lemma}
\label{splfibre}
$\xx_s$ is isomorphic to $\xx'_s$ over $\Spec (k)$.    
\end{lemma}

\begin{proof}
As the question is local, we can consider $\xx$ be affine, say $\Spec (A)$ where $A$ is a finite $\oo$-algebra. We know there is an isomorphism $$(A \otimes_{\oo} \oo') \otimes_{\oo'} k \cong A \otimes_{\oo} k.$$ Hence, the special fibre $\xx_s$ of $\xx$ is isomorphic to $\xx'_s$.
\end{proof}

\begin{proof}[Proof of \autoref{thma}]
Let $P \in \xx(K)$ be a $f$-periodic point of period $n$ where $n$ is coprime to $p$. We can assume $P \in \xx(K')$ where $K' \subset K$ is a local field of $F$. By \autoref{naf_lemma} and remark after that, $n$ is bounded and the bound depends only on the special fibre $\xx_s$. But by \autoref{splfibre} all base changes have isomorphic special fibre and this completes the proof.
\end{proof}

\begin{remark}
\label{remark1}
The following example shows the failure of boundedness for all periods : Let $F = \QQ_p$ and $K = \QQ_p(\mu_{p^{\infty}})$ for some odd prime $p$. Let us consider $\xx = \mathbb{P}^1_{\ZZ_p}$ and $f : X \to X$ by $$[x : y] \mapsto [x^q : y^q]$$ for some odd prime $q \ne p$. Clearly all $p$-power roots of unity are periodic points. Hence, $p$-power part of the period is not bounded.
\end{remark}

\subsection{Corollary to number field}

Assuming the setup of \autoref{thmb}, let $F_\pp$ be the completion of $F$ at $\pp$ and $\oo'$ be the ring of integers of $F_\pp$. 

\begin{proof}[Proof of \autoref{thmb}]
Let $P \in X(K)$ be a $f$-periodic point of period $n$ where $n$ is prime to $p$. We can assume $P \in \xx (K')$ where $K' \subset K$ is a number field over $F$. We can consider that $\xx$ is defined over $\Spec (\oo')$. As $\xx(K') \subset \xx(K'_\pp)$, following the same argument for the proof of \autoref{thma} completes the proof.
\end{proof}

\subsection{Application to abelian varieties}

\begin{proof}[Proof of \autoref{ablian}]
Let us fix a prime $q$. For a prime $\ell \ne q$, $P \in A(K)[\ell]$ is a periodic point under the isogeny $[q] : A \to A$. Suppose the period of $P$ is $mp^n$ where $(m, p) = 1$. Then we get $$q^{mp^n} \, P = P$$ Hence $q^{mp^n} \equiv 1 \; (\modulo \ell)$. 

Now suppose $\ell \not \equiv 1 \; (\modulo p^{a})$ where $a \ge 1$, then $n \le a - 1$. By \autoref{thma}, for each $b \le a - 1$, there are only finitely many choices for $m$ and for one such $m$ there are finitely many primes $\ell$ satisfying $q^{mp^b} \equiv 1 \, (\modulo \ell)$. Thus, we get $\TT(A, K) \cap \set{ \ell \, {\rm prime} \, : \, \ell \not \equiv 1 \; (\modulo p^a)}$ is finite. This completes the proof.
\end{proof}

\begin{remark}
It is easy to show that if an elliptic curve $E$ has good reduction, then the set $\TT(E, K)$ is finite.
\end{remark}

\begin{proposition}
Let $E$ be an elliptic curve over $\QQ_p$ with bad reduction and $$K_0 \subset K_1 \subset K_2 \subset \dots$$ be a tower of totally ramified local fields and $K_{n+1}$ over $K_n$ is purely wildly ramified for large $n$. Then $\TT(E, \cup K_n)$ is finite.
\end{proposition}

\begin{proof}
Let $E'$ be the reduction of the elliptic curve modulo $p$. For each $n$, there is a map $$\phi_n : E(K_n) \to E'(\FF_p)$$
Let $E_0 (K_n) = \phi_n^{-1} (E'_{\ns} (\FF_p))$ where $E'_{\ns}$ is the non-singular part of the reduction $E'$.  By \cite[Proposition VII.2.1]{aec}, $E_0 (K_n)$ is a subgroup of $E(K_n)$ and the map $E_0 (K_n) \to E'_{\ns} (\FF_p)$ is surjective. Hence we can consider an exact sequence like $$0 \to E_1 (K_n) \to E_0 (K_n) \to E'_{\ns} (\FF_p) \to 0$$ where $E_1 (K_n) = \ker \left( E_0 (K_n) \to E'_{\ns} (\FF_p) \right)$. 

By \cite[Proposition VII.3.1]{aec}, $E_1 (K_n)$ does not contain any prime-to-$p$ torsion and by theory of N\'{e}ron models, we can get $E(K_n) / E_0 (K'_n)$ is cyclic of order $v_n (\Delta)$ in the case of split multiplicative reduction or otherwise finite of order less than equal $4$ \cite[Theorem VII.6.1]{aec}. Here $\Delta$ is the discriminant of $E$ and $v_n$ is the $p$-adic valuation in $K_n$. Since prime-to-$p$ factor of $v_n(\Delta)$ is stable for large $n$, this completes the proof.
\end{proof}

\begin{remark}
This result can be proved for any local $F$ instead of $\QQ_p$. Hence, this proposition covers the elliptic curve case of the result by Serre \cite{serre} and Imai \cite{imai}.
\end{remark}

Now, we can raise the following question :

\begin{question}
If each $K_n$ is tamely ramified over $\QQ_p$, can the set $\TT (E, \cup K_n)$ be infinite?
\end{question}

\subsection{Finiteness of periodic points}

Suppose $X$ is a projective variety and $f : X \to X$ be a self map satisfying the hypothesis of \autoref{thma}. We will assume there exists a line bundle $\LL$ on $X$ such that $f^{*} (\LL) \otimes \LL^{-1}$ is an ample line bundle. We will get similar results like Fakhruddin \cite[Lemma 2]{naf}. 

\begin{corollary}
The set of $K$-rational $f$-periodic points with prime-to-$p$ periods is finite.
\end{corollary}

\begin{proof}
We can assume $f$ to be finite. Let $Y$ be the closure of all $f$-periodic points with prime-to-$p$ period. Clearly there is a positive integer $m$ such that $f^m$ is the identity on $Y$. Now $$(f^m)^{*} (\LL) \otimes \LL^{-1} = \bigotimes_{i=1}^{m-1} (f^i)^{*} (f^{*} (\LL) \otimes \LL^{-1})$$ is also ample, but their restriction to $Y$ is trivial. Hence, $Y$ must be $0$-dimensional.
\end{proof}

\section{Remarks on an analogue of \autoref{thma}}

\begin{proposition}
\label{ffields}
Let $\oo$ be a discrete valuation ring with residue field $k$ of characteristic $0$. We consider $\xx$ to be a separated scheme of finite type over $\Spec (\oo)$ and a self-map $f : \xx \to \xx$ be defined over $\Spec (\oo)$. Let $p$ be a rational prime and $\Per_{(p)} (\xx_s (k), f_s)$ is finite. Suppose for every $d \ge 1$ all extensions of degree $d$ over $k$ contain only finitely many $n$-th roots of unity where $n$ is prime to $p$. Then, $\Per_{(p)} (\xx(\oo), f)$ is finite.
\end{proposition}

\begin{proof}
Let $P \in \xx(\oo)$ have $f$-period $n$ where $n$ is coprime to $p$. Let $P_s \in \xx_s(k)$ be the reduction of $P$ at the special fibre. Let $r$ be the period of $P_s$. Then $r \mid n$ and we can write $n = r \cdot m$. Hence $P_s$ is a fixed point for $g = f^r$. 

Let $Z$ be the reduced Zariski closure of $g$-orbit of $P$. We can consider $Z = \Spec(A)$ where $A$ is reduced torsion-free local $\oo$-algebra and $g$ induces an $\oo$-endomorphism of $A$ [cf. \autoref{naf_lemma}]. Let $\mm$ be the maximal ideal of $A$. Clearly $g(\mm) \subset \mm$ and $g$ induces a $k$-endomophism of $\mm / \mm^2$.

Let the order of $g$ in $\End_k (\mm/\mm^2)$ be $s$. Then clearly, $s \mid m$ and we can write $m = s \cdot m'$. Following the argument by \cite[Proposition 1]{naf}, we can bound $s$ by some $\delta > 0$ which only depends on the scheme $\xx$ and the field $k$. Then $g^s$ becomes a unipotent map with respect to the filtration of $A$ by powers of $\mm$ and $g^s$ must be identity as $k$ is of characteristic 0. Hence $$n = r \cdot s \le B \cdot \delta$$ where $B$ is the bound of the set $\Per_{(p)} (\xx_s (k), f_s)$ and this completes the proof.
\end{proof}  

Let $F$ be a local field $\QQ_p$. Let $L$ be a finite extension over $F$ which is totally ramified. Let $\mu_{(p)} (L)$ be the set of $n$-th roots of unity in $L$ where $n$ is coprime to $p$.  It is easy to show that the set $\mu_{(p)} (L)$ is bounded independent of $L$. Indeed, $[F : \QQ_p]^2$ works as a bound for all possible $n$-th roots of unity that can occur. If $K$ is an infinite extension over $F$ which is totally ramified, the set $\mu_{(p)}(K)$ will be bounded. 
 
Suppose, $X$ is a separated scheme of finite type over $K \left[ [T] \right]$ and $X_s$ is the special fibre over $K$. Let $f: X \to X$ be a map such that the map restricted to the special fibre $f_s : X_s \to X_s$ satisfies the hypothesis of \autoref{thma}. Then using \autoref{ffields}, we get 

\begin{theorem}
$\Per_{(p)} (X(K \left[ [T] \right]), f)$ is finite.
\end{theorem}

\section{Remarks on some other infinite extensions}
\label{others}

Let $F$ be a number field and $F^{(n)}$ be the compositum of all extensions $K$ over $F$ of degree at most $n$. Clearly, $F^{(n)}$ is Galois over $F$. Suppose $\pp$ is a prime in $F$ and $\pp'$ be any prime in $F^{(n)}$ lying above $\pp$. Then, it can be shown that their local degrees $[F^{(n)}_{\pp'} : F_{\pp}]$ are bounded \cite[Proposition 1]{bz01}. Suppose, $X$ is a proper variety over $\oo_F$ where $\oo_F$ is the ring of integers of $F$ and $f : X \to X$ be a self map. Then for sufficiently large two rational primes, \autoref{naf_lemma} gives us

\begin{theorem}
$\Per (X(F^{(n)}), f)$ is finite.
\end{theorem}

\begin{remark}
The previous result also can be proved using \cite[Theorem 2]{naf} for a sufficiently large prime.  
\end{remark}

In general, let us assume $K$ be a Galois extension over a number field $F$. Let $\pp$ be a prime in $F$ and $\pp'$ be any prime of $K$ lying over $\pp$. If the local degrees $[K_{\pp'} : F_\pp]$ are finite for infinitely many primes, then we can prove the previous result for $K$-rational points.

\begin{remark}
It is enough to have finiteness of local degrees for primes over two sufficiently large rational primes.
\end{remark}

\bibliographystyle{amsalpha}
\bibliography{ref}

\end{document}